\documentclass[12pt,reqno]{amsart}
\usepackage{amssymb,amsmath,graphicx,
	amsfonts,euscript}
\usepackage{color}
\usepackage{mathrsfs}

\newcommand{\q}{{\mathcal{Q}}}
\newcommand{\p}{{\mathcal{P}}}
\newcommand{\R}{{\mathbb{R}}}
\def\div{ \hbox{\rm div}\,  }

\def\A{\mathcal{A}}
\newcommand{\Z}{{\mathbb{Z}}}

\def\aa{ \phi }
\def\ta{ \theta }

\def\nn{\nonumber}

\def\R{{\mathbb R}}
\def\div{ \hbox{\rm div}\,  }
\def\Z{{\mathbb{Z}}}
\def\u{\mathbf{u}}
\def\v{\varphi}
\def\u{\mathbf{u}}

\def\ddj{\dot \Delta_j}

\newcommand{\norm}[1]{\lVert #1 \rVert}
\theoremstyle{plain}

\newtheorem{theorem}{Theorem}[section]
\newtheorem{lemma}[theorem]{Lemma}
\newtheorem{definition}[theorem]{Definition}

\newtheorem{remark}[theorem]{Remark}
\numberwithin{equation}{section}

\begin{document}

\title[Global solutions to CNS system with potential temperature transport]{Global strong solutions to the compressible Navier-Stokes system with potential\\ temperature transport}

\author[Xiaoping Zhai,  Yongsheng Li  and Fujun Zhou]{Xiaoping Zhai$^\dag$, Yongsheng Li$^\ddag$  and Fujun Zhou$^\ddag$}

\address{$^\dag$ School  of Mathematics and Statistics, Guangdong University of Technology,
	Guangzhou, 510520,   China}

\email{pingxiaozhai@163.com (X. Zhai)}

\address{$^\ddag$ School of Mathematics,
South China University of Technology,
Guangzhou, 510640, China}

\email{yshli@scut.edu.cn (Y. Li)}
\email{fujunht@scut.edu.cn (F. Zhou)}

\begin{abstract}
We study the global strong solutions to  the compressible Navier-Stokes system with potential temperature transport in $\mathbb{R}^n.$ Different from  the Navier-Stokes-Fourier
system,  the pressure is a nonlinear
function of the density and the potential temperature,
 we can not exploit the special quasi-diagonalization structure of this system to capture any dissipation of the density.
Some new idea and delicate analysis involved in high or low frequency  decomposition in the Besov spaces have to be made to close the energy estimates.
	\end{abstract}
\maketitle
\section{ Introduction and the main results}
In this paper, we are concerned with the Cauchy problem of the compressible Navier-Stokes system with potential temperature transport. The system has the following form:
\begin{eqnarray}\label{mm1}
\left\{\begin{aligned}
&\partial_t\rho+\div(\rho \u)=0,\\
&\rho(\partial_t\mathbf{u} + \u\cdot \nabla \u) -  \mu \Delta \u -(\lambda+\mu) \nabla \div \u+\nabla P(\rho,\ta) =  0,\\
&\partial_t(\rho\ta)+\div(\rho\ta \u)=0,\\
 &(\rho,\u,\theta)|_{t=0}=(\rho_0,\u_0,\theta_0).
\end{aligned}\right.
\end{eqnarray}
Here, $x=(x_1,x_2,\cdot\cdot\cdot,x_n)\in \R^n$ and $t\geq 0$ are the space and time variables, respectively.
The unknown functions $\rho$ is the fluid density,
 $\u$ is the velocity field, $ {P}$ is the scalar pressure, $\ta$ is the fluid potential temperature.
The  viscosity coefficients $\mu$ and $\lambda$ are subject to  the standard  physical restrictions:
\begin{align*}
\mu>0\quad\hbox{and}\quad
n\lambda+2 \mu>0.
\end{align*}
$\gamma>1$ is the
adiabatic index, the pressure state equation reads
\begin{align}\label{yali}
P(\rho,\ta)=A(\rho\ta)^\gamma,\quad A>0.
\end{align}

System \eqref{mm1} with \eqref{yali} governs the motion of viscous compressible fluids
with potential temperature, where diabatic processes and the influence of molecular transport on
potential temperature are excluded. It's often used in meteorological applications, see, e.g., \cite{klein}, \cite{lions} and the references therein.  Due to the importance of the system, it has  drawn  much attention recently.
For any $\gamma > 3/2, n = 3,$
Mich\'alek \cite{mic} studied the stability of weak solutions   of  \eqref{mm1}--\eqref{yali}, see also \cite[Chapter 5 and Chapter 8]{lions},
in which Lions investigated the stability of weak solutions for the compressible Navier-Stokes equations with a scalar
transport for $ \gamma\ge9/5$.
For any $ \gamma\ge9/5, n = 3$ or $ \gamma>1, n = 2$,
Maltese {\it et al.} \cite{mal} obtained the existence of global-in-time weak solutions to \eqref{mm1}--\eqref{yali}  with $\ta^\gamma=s^\gamma$ ($s$ is  the entropy).
Feireisl {\it et al.} \cite{fei1} analyzed the singular limit in the low Mach/Froude
number regime of the above Navier-Stokes system with $\gamma>3/2$.
By
analyzing
the convergence of a suitable numerical scheme, the mixed finite element-finite volume method,
Luk\'a$\check{\mathrm{c}}$ov\'a-Medvid'ov\'a and Sch\"omer \cite{luk1}
proved the global-in-time existence of DMV (dissipative measure-valued) solutions for all adiabatic indices $\gamma > 1$ for $n = 2, 3$. Later, they \cite{luk2} further established a DMV-strong uniqueness result for \eqref{mm1}--\eqref{yali}. Moreover they showed that strong solutions are stable in the class of DMV solutions.
However, to  the authors' knowledge, there are few results about the strong solutions of \eqref{mm1} with \eqref{yali}.

If the effect of the temperature is neglected and thus  the pressure is a  function of $\rho$,  Eq. (\ref{mm1}) reduces to the following isentropic  compressible Navier-Stokes equations,
\begin{eqnarray}\label{cns}
\left\{\begin{aligned}
&\partial_t\rho+\div(\rho \u)=0,\\
&\rho(\partial_t\mathbf{u} + \u\cdot \nabla \u) -  \mu \Delta \u -(\lambda+\mu) \nabla \div \u+\nabla \rho^\gamma =  0.
\end{aligned}\right.
\end{eqnarray}
 The above system \eqref{cns}
 has been widely studied (see  \cite{charve}, \cite{haspot}, \cite{chenqionglei}--\cite{danchin2000}, \cite{helingbing}--\cite{zhaixiaopingsiam} and the references therein).
For arbitrary initial data and $\bar{\rho}=0$, the breakthrough was made by Lions \cite{lions}, where the author proved the global existence of weak solutions for $P=A\rho^\gamma$ for $\gamma\ge 9/5.$  Later,
Feireisl {\it et al.} \cite{ferisal} extended Lions' result to the case of $\gamma> 3/2$. Jiang and Zhang \cite{zhangping1}, \cite{zhangping2} improved the global existence of weak solution for any $\gamma>1$ for the spherically symmetric or axisymmetric initial data.
However, the question of the regularity and uniqueness
 of weak solutions is completely open even in the case of two dimensional space.

Compared to the weak solutions, there are relative fruitful  results on the strong solutions.
Nash \cite{nash} proved the local existence and uniqueness of smooth solution of the isentropic compressible Navier-Stokes equations for smooth initial data without vacuum.
The global classical solutions were first
obtained by Matsumura and Nishida \cite{mat} for initial data $(\rho_0, \u_0)$ close to an  equilibrium $(\bar{\rho}, 0)$ in $H^3\times H^3(\R^3)$, $\bar{\rho}>0$.
This result  was further generalized by Huang {\it et al.} 
 \cite{huangxiangdi}  with constant state as far field which could be either vacuum or non-vacuum in $\R^3$ with smooth initial data. Moreover, the initial data are of small total energy but possibly large oscillations, see also Li and Xin \cite{li2019}
for further developments.

In the  framework of  critical spaces, a breakthrough was made by
Danchin \cite{danchin2000} for the  isentropic  compressible
Navier-Stokes equations, where the author proved the local wellposedness with large initial data and global solutions with small initial data.
This result was further extended by  Charve and Danchin \cite{charve},
 Chen {\it et al.} \cite{chenqionglei}, Haspot \cite{haspot}. 
  Recently, more and more researchers are devoted to the global solutions of the compressible Navier-
Stokes equations with  different class of large initial data in the  critical spaces.
Based on the spectral analysis for the linearized system and Hoff's energy method, Wang {\it et al.} \cite{wangchao} proved a global existence result of three dimensional compressible
Navier--Stokes equations for a class of initial data, which may have large oscillation for the density and large energy for
the velocity.
Making  use of the dispersive estimates for  the system of acoustics, Fang {\it et al.} \cite{fangdaoyuan} constructed the global strong solutions to \eqref{cns} in $\R^n$ which allows
 the low frequency part of the initial velocity field be large.
He {\it et al.} \cite{huangjingchi} studied the global-in-time stability of
\eqref{cns} and proved that any perturbed solution remains close to the reference
solution if they are initially close to each other. As a byproduct, they
constructed also the global large solutions to \eqref{cns} which allow the vertical
component of the initial velocity to be arbitrarily large.
 Zhai {\it et al.} \cite{zhaixiaopingsiam} also
constructed  global solutions in $\R^3$ with another class of large initial data satisfying  nonlinear smallness which allows the each
component of the incompressible part of  initial velocity  could be arbitrarily large.
Let $\p\stackrel{\mathrm{def}}{=}\mathcal{I}-\q=\mathcal{I}-\nabla\Delta^{-1}\div$ be the stokes projection operator.
Recently, Danchin and Mucha \cite{danchin2018} obtained
the global existence of  regular solutions to \eqref{cns}
 with arbitrary large initial incompressible velocity $\p \u_0$,  almost  constant
density $\rho_0$, and   large  viscosity $\lambda$.
This result was further extended by Chen and Zhai  \cite{zhaixiaopingjmfm2019}
in a critical $L^p$  framework, which implies that the highly oscillating
initial data are allowed.

Our main goal is to solve the Cauchy problem of the compressible Navier-Stokes system with potential temperature transport.
We will  concentrate on the   local well-posedness issue for large data with no vacuum,
on   the global well-posedness issue for small perturbations of a constant steady equilibrium, in the critical regularity framework.
By criticality, we mean that  the scaling transformation which keeps the norms of the function space of the solution also
leaves \eqref{mm1} invariant. In the case
of compressible fluids, it is easy to see that the transformation
$$(\rho(t,x), \u(t,x),\ta(t,x))\mapsto (\rho(\ell^2t,\ell x),
\ell \u(\ell^2t,\ell x),\ta(\ell^2t,\ell x)), \quad \ell>0,$$
possesses such property provided the pressure term has been changed accordingly.
One can check that
the product space $\dot{B}_{2,1}^{\frac {n}{2}}({\mathbb R} ^n)\times \dot{B}_{2,1}^{\frac {n}{2}-1}({\mathbb R} ^n) \times\dot{B}_{2,1}^{\frac {n}{2}}({\mathbb R} ^n)$ is critical spaces  for the system (\ref{mm1}).

To overcome the difficulties arising from the strong nonlinearity of the pressure, combining with the renormalized equations of $\rho$ and $\theta$,  we will transform \eqref{mm1} into another form  in terms of variables $\rho,\u,P$,
\begin{eqnarray}\label{mm2}
\left\{\begin{aligned}
&\partial_t\rho+\div(\rho \u)=0,\\
&\rho(\partial_t\mathbf{u} + \u\cdot \nabla \u) -  \mu \Delta \u -(\lambda+\mu) \nabla \div \u+\nabla P =  0,\\
&\partial_tP+\u\cdot \nabla P +\gamma P\div \u=0 .
\end{aligned}\right.
\end{eqnarray}
Moreover, to deal with the  lack of dissipation on $\rho$ and $\ta$, we shall restrict the initial data of \eqref{mm2} to satisfy
\begin{align*}
(\rho,\u,P)(0,x)=(\rho_0,\u_0,P_0)(x)\rightarrow (\bar{\rho},\mathbf{0},\bar{P})\quad\hbox{as} \quad {|x|\rightarrow \infty},
\end{align*}
where $\bar{\rho}$ and $\bar{P}$
are  two positive constants.
For convenience, denote $\mu=1,\lambda=0$, $\bar{\rho}=\bar{P}=1$
and define
\begin{align*}
\mathcal{A}\u\stackrel{\mathrm{def}}{=} \Delta \u +\nabla \div \u,\quad\rho\stackrel{\mathrm{def}}{=}1+a,\quad P\stackrel{\mathrm{def}}{=}1+b,
\quad\hbox{and}\quad I(a)\stackrel{\mathrm{def}}{=}\frac{a}{1+a},
\end{align*}
 we can rewrite \eqref{mm2} into the following new form
\begin{eqnarray}\label{mm3}
\left\{\begin{aligned}
&\partial_ta+\div \u+\u\cdot\nabla a+a\div \u=0,\\
&\partial_t\u + \u\cdot \nabla \u -  \A \u +\nabla b =I(a)\nabla{b}-I(a)\A \u,\\
&\partial_tb+\gamma\div \u+\u\cdot\nabla b+\gamma b\div \u=0,\\
&(a,\u,b)|_{t=0}=(a_0,\u_0,b_0).
\end{aligned}\right.
\end{eqnarray}
The above system \eqref{mm3} has a similar structure to  the isentropic compressible Navier-Stokes equations if we  regard $a$ and $b$  as a whole, hence, we can follow the method used in \cite{chenqionglei}, \cite{danchin2000} and \cite{danchin2014} to prove the local wellposedness. For convenience to the readers, we only state the theorem
 as follows without detailed proof.
\begin{theorem}(Local wellposedness)\quad \label{dingli1}
Let   $n\ge2$. Then for any  $\u_0\in \dot{B}_{2,1}^{\frac n2-1}(\R^n)$,  $(a_0,b_0)\in \dot{B}_{2,1}^{\frac n2}(\R^n)$ with $1 + a_0$ bounded away from zero,  there exists a positive time
$T$ such that the system \eqref{mm3} has a unique solution with
\begin{align*}
&(a,b)\in C([0,T ];{\dot{B}}_{2,1}^{\frac {n}{2}}),\quad \u\in C([0,T ];{\dot{B}}_{2,1}^{\frac {n}{2}-1})\cap L^{1}
([0,T];{\dot{B}}_{2,1}^{\frac n2+1}).
\end{align*}
\end{theorem}
\begin{remark}
We can generalize the above theorem to an $L^p$-critical framework as  \cite{chenqionglei} and \cite{danchin2014}.
\end{remark}

Before stating the second theorem, we introduce some notations.
Let $\mathcal{S}(\R^n)$ be the space of
rapidly decreasing functions over $\R^n$ and $\mathcal{S}'(\R^n)$ be its dual
space. For any $z \in\mathcal{S}'(\R^n)$,
the lower and higher frequency parts are expressed as
\begin{align*}
z^\ell\stackrel{\mathrm{def}}{=}\sum_{j\leq j_0}\ddj z\quad\hbox{and}\quad
z^h\stackrel{\mathrm{def}}{=}\sum_{j>j_0}\ddj z
\end{align*}
for some fixed   integer $j_0\ge 1$ (the value of $j_0$ is dependent in  the proof of the main theorems). 
The corresponding  truncated semi-norms are defined  as follows:
\begin{align*}
\norm{z}^\ell_{\dot B^{s}_{2,1}}
\stackrel{\mathrm{def}}{=}  \norm{z^{\ell}}_{\dot B^{s}_{2,1}}
\ \hbox{ and }\   \norm{z}^{h}_{\dot B^{s}_{2,1}}
\stackrel{\mathrm{def}}{=}  \norm{z^{h}}_{\dot B^{s}_{2,1}}.
\end{align*}

The second main result of the paper is stated as follows.
\begin{theorem}(Global  well-posedness)\quad \label{dingli2}
Let   $n\ge2$.
Then for any  $(a_0^\ell,\u_0,b_0^\ell)\in \dot{B}_{2,1}^{\frac n2-1}(\R^n)$,   and  $(a^h_0, b_0^h)\in \dot{B}_{2,1}^{\frac n2}(\R^n)$,
 there
exists a positive constant $c_0$ such that if,
\begin{align}\label{smallness}
\norm{(a_0^\ell,{\u}_{0},b_0^\ell)}_{\dot{B}^{\frac{n}{2}-1}_{2,1}}+
\norm{({a}^h_{0},{b}^h_{0})}_{\dot{B}^{\frac{n}{2}}_{2,1}}\leq c_0,
\end{align}
the system \eqref{mm3} has a unique global solution $(a,\u,b)$ so that
\begin{align*}
&a^\ell\in C_b(\R^+;{\dot{B}}_{2,1}^{\frac {n}{2}-1}),\quad a^h\in C_b(\R^+;{\dot{B}}_{2,1}^{\frac {n}{2}}),\nn\\
 &b^h\in C_b(\R^+;{\dot{B}}_{2,1}^{\frac {n}{2}})\cap L^{1}
(\R^+;{\dot{B}}_{2,1}^{\frac n2}),\nn\\
&(b^\ell,\u)\in C_b(\R^+;{\dot{B}}_{2,1}^{\frac {n}{2}-1})\cap L^{1}
(\R^+;{\dot{B}}_{2,1}^{\frac n2+1}).
\end{align*}
Moreover,  there exists some constant $C$ such that
\begin{align*}
&\norm{(a^\ell,b^{\ell}, \u)}_{\widetilde{L}^{\infty}_t(\dot{B}^{\frac{n}{2}-1}_{2,1})}+\norm{(a^h,b^h)}_{\widetilde{L}^{\infty}_t(\dot{B}^{\frac{n}{2}}_{2,1})}
+\norm{(b^{\ell}, \u)}_{L^1_t(\dot{B}^{\frac{n}{2}+1}_{2,1})}+\norm{b^h}_{{L}^{1}_t(\dot{B}^{\frac{n}{2}}_{2,1})}\leq Cc_0.
\end{align*}
\end{theorem}
\begin{remark}
Due to lack of dissipation of the density, it may be a challenge to generalize the above theorem to more general Besov spaces
related to the $L^p$  with $p\not=2.$
\end{remark}
\subsection*{Strategy of the proof of Theorem \ref{dingli2}}
Now let us outline the main points of the study and explain some of the major difficulties and techniques presented in this article. By the continuity argument, the existence of the global  solutions can be proven by combining the local existence and the {\it a priori} estimates. The local well-posedness can be proven by the standard compact argument as \cite{danchin2000}. The key point is to obtain the {\it a priori} estimates of the strong solutions. More specifically, since the dissipative variables  $b$ and  $\u$ satisfy $\eqref{mm3}_3$ and $\eqref{mm3}_2$ whose linear parts possess the same structure as that of the compressible isentropic Navier-Stokes equations \eqref{cns}. By  spectral analysis just like \cite{chenqionglei},  \cite{zhaixiaopingjmfm2019}, \cite{helingbing}, the variable  $b$ has smoothing effect in the low frequency and  damping effect in the high frequency. Moreover,
 the uniform bound of  $b,\u$ can be established by a direct energy method as in \cite{chenqionglei}--\cite{danchin2000}. Due to the appearance of the non-dissipative variable $a$, the main difficulty lies in how to enclose  the energy estimates of the variable $a$.
Different from the isentropic compressible Navier-Stokes equations \eqref{cns}, there is a missing term $\nabla a$ in the momentum equation, so,
 the density function $a$ possesses neither smoothing effect in the low frequency  nor damping effect in the high frequency. Moreover,  we even cannot directly control the low frequency part of $a$
in the spaces $\dot{B}^{\frac{n}{2}-1}_{2,1}(\R^n)$ as the linear term $\div \u$ appeared in the first equation of \eqref{mm3}. Indeed, for the first equation of \eqref{mm3}, if we make the energy estimate for $a^\ell$ in the space ${\widetilde{L}^{\infty}_t(\dot{B}^{\frac{n}{2}-1}_{2,1})}$, we  have to control the linear term $\|\u^\ell\|_{{L}^{1}_t(\dot{B}^{\frac{n}{2}}_{2,1})}$. However, we can only obtain $\|\u^\ell\|_{{L}^{1}_t(\dot{B}^{\frac{n}{2}+1}_{2,1})}$ from the previous energy argument of $(b^\ell,\u^\ell).$
 This leads to the loss of control for the nonlinear terms $I(a)\nabla{b}$ and $I(a)\A \u$.
To overcome this difficulty, we need to introduce a combination function $\phi\stackrel{\mathrm{def}}{=}\gamma a-b$ to annihilate the liner term $\div \u$.  Exploiting some delicate energy estimates, we can first get
 the bound of $\norm{\aa^{\ell}}_{\widetilde{L}^{\infty}_t(\dot{B}^{\frac{n}{2}-1}_{2,1})}$  and then obtain the control of $\norm{b^{\ell}}_{\widetilde{L}^{\infty}_t(\dot{B}^{\frac{n}{2}-1}_{2,1})}$, which further  implies the bound of $\norm{a^{\ell}}_{\widetilde{L}^{\infty}_t(\dot{B}^{\frac{n}{2}-1}_{2,1})}$.
This enables us to obtain the energy estimates of the non-dissipative variable $a$ and thus to prove the global wellposedness.
\medskip

The rest of this paper is arranged as follows. In the second section, we  recall some basic facts about Littlewood-Paley theory. In the third section, we use four subsections to prove Theorem \ref{dingli2}.
 In the first subsection,  we obtain the estimates   of  $\p\u$.
 In the second and third subsections, we obtain the estimates of compressible part of $(a,b,\q\u)$ in the low frequency and  high frequency, respectively.
 In the last subsection, we use the continuity argument to close  the energy estimates
and thus complete the proof of Theorem \ref{dingli2}.

\medskip
Let us introduce some  notations.
For two operators $A$ and $B$, we denote $[A,B]=AB-BA$, the commutator between $A$ and $B$.
The letter $C$ stands for a generic constant whose meaning is clear from the context.
We denote $\langle a,b\rangle$ the $L^2(\R^n)$ inner product of $a$ and $b$
 and write $a\lesssim b$ instead of $a\leq Cb$.
Given a Banach space $X$, we shall denote $\norm{(a,b)}_{X}=\norm{a}_{X}+\norm{b}_{X}$.

\medskip


\section{ Preliminaries }
For  readers' convenience, in this section, we list some basic knowledge on  Littlewood-Paley  theory.
The {Littlewood-Paley decomposition}  plays a central role in our analysis.
To define it,   fix some  smooth radial non increasing function $\chi$
supported in the ball $B(0,\frac 43)$ of $\R^n,$ and with value $1$ on, say,   $B(0,\frac34),$ then set
$\psi(\xi)=\chi(\frac{\xi}{2})-\chi(\xi).$ We have
$$
\qquad\sum_{j\in\Z}\psi(2^{-j}\cdot)=1\ \hbox{ in }\ \R^n\setminus\{0\}
\quad\hbox{and}\quad \mathrm{Supp}\,\psi\subset \Big\{\xi\in\R^n : \frac34\leq|\xi|\leq\frac83\Big\}\cdotp
$$
The homogeneous dyadic blocks $\dot{\Delta}_j$ are defined on tempered distributions by
$$\dot{\Delta}_j u\stackrel{\mathrm{def}}{=}\psi(2^{-j}D)u\stackrel{\mathrm{def}}{=}{\mathcal F}^{-1}(\psi(2^{-j}\cdot){\mathcal F} u).
$$
Let us remark that, for any homogeneous function $A$ of order 0 smooth outside 0, we have
\begin{equation*}\label{}
\forall p\in[1,\infty],\quad\quad\norm{\ddj (A(D) u)}_{L^p}\le C\norm{\ddj u}_{L^p}.
\end{equation*}
\begin{definition}
Let $p,r$ be in~$[1,+\infty]$ and~$s$ in~$\R$, $u\in\mathcal{S}'(\R^n)$. We define the Besov norm by
$$
\norm{u}_{{\dot{B}^s_{p,r}}}\stackrel{\mathrm{def}}{=}\norm{(2^{js}\|\ddj
u\|_{L^{p}})_j}_{\ell ^{r}({\mathop{\mathbb Z\kern 0pt}\nolimits})}.
$$
We then define the spaces
$\dot{B}_{p,r}^s\stackrel{\mathrm{def}}{=}\left\{u\in\mathcal{S}'_h(\R^n),\big|
\norm{u}_{\dot{B}_{p,r}^s}<\infty\right\}$, where $u\in \mathcal{S}'_h(\R^n)$ means that $u\in \mathcal{S}'(\R^n)$ and $\lim_{j\to -\infty}\norm{\dot{S}_ju}_{L^\infty}=0$ (see Definition 1.26 of \cite{bcd}).
\end{definition}

When employing parabolic estimates in Besov spaces, it is somehow natural to take the time-Lebesgue norm before performing the summation for computing the Besov norm. So we next introduce the following Besov-Chemin-Lerner space $\widetilde{L}_T^q(\dot{B}_{p,r}^s)$ (see\,\cite{bcd}):
$$
\widetilde{L}_T^q(\dot{B}_{p,r}^s)={\Big\{}u\in (0,+\infty)\times\mathcal{S}'_h(\R^n):
\norm{u}_{\widetilde{L}_T^q(\dot{B}_{p,r}^s)}<+\infty{\Big\}},
$$
where
$$
\norm{u}_{\widetilde{L}_T^q(\dot{B}_{p,r}^s)}\stackrel{\mathrm{def}}{=}\big\|{2^{ks}\|\dot{\Delta}_k u(t)\|_{L^q(0,T;L^p)}}\big\|_{\ell^r}.
$$
The index $T$ will be omitted if $T=+\infty$ and we shall denote by $\widetilde{\mathcal{C}}_b([0,T]; \dot{B}^s_{p,r})$ the subset of functions of $\widetilde{L}^\infty_T(\dot{B}^s_{p,r})$ which are also continuous from
$[0,T]$ to $\dot{B}^s_{p,r}$.

The following Bernstein's lemma will be repeatedly used throughout this paper.

\begin{lemma}\label{bernstein}
Let $\mathcal{B}$ be a ball and $\mathcal{C}$ a ring of $\mathbb{R}^n$. A constant $C$ exists so that for any positive real number $\lambda$, any
non-negative integer k, any smooth homogeneous function $\sigma$ of degree m, and any couple of real numbers $(p, q)$ with
$1\le p \le q\le\infty$, there hold
\begin{align*}
&&\mathrm{Supp} \,\hat{u}\subset\lambda \mathcal{B}\Rightarrow\sup_{|\alpha|=k}\norm{\partial^{\alpha}u}_{L^q}\le C^{k+1}\lambda^{k+n(\frac1p-\frac1q)}\norm{u}_{L^p},\\
&&\mathrm{Supp} \,\hat{u}\subset\lambda \mathcal{C}\Rightarrow C^{-k-1}\lambda^k\norm{u}_{L^p}\le\sup_{|\alpha|=k}\norm{\partial^{\alpha}u}_{L^p}
\le C^{k+1}\lambda^{k}\norm{u}_{L^p},\\
&&\mathrm{Supp} \,\hat{u}\subset\lambda \mathcal{C}\Rightarrow \norm{\sigma(D)u}_{L^q}\le C_{\sigma,m}\lambda^{m+n(\frac1p-\frac1q)}\norm{u}_{L^p}.
\end{align*}
\end{lemma}

Next, we give the important product acts on Besov spaces.
\begin{lemma}\label{law}{\rm(\cite{danchin2000})}
Let  $s_1\leq \frac{n}{2}$, $s_2\leq \frac n2$ and $s_1+s_2>0$. For any $u\in\dot{B}_{2,1}^{s_1}({\mathbb R} ^n)$, $v\in\dot{B}_{2,1}^{s_2}({\mathbb R} ^n)$, we have
\begin{align*}
\norm{uv}_{\dot{B}_{2,1}^{s_1+s_2 -\frac{n}{2}}}\lesssim \norm{u}_{\dot{B}_{2,1}^{s_1}}\norm{v}_{\dot{B}_{2,1}^{s_2}}.
\end{align*}
\end{lemma}
\begin{lemma}(Lemma 2.100 in \cite{bcd})\label{jiaohuanzi}
Let  $-1-\frac n2<s\leq 1+\frac n2$,
$v\in \dot{B}_{2,1}^{s}(\R^n)$ and $u\in \dot{B}_{2,1}^{\frac{n}{2}+1}(\R^n)$ with $\div u=0$.
Then there holds
$$
\norm{[\dot{\Delta}_j,u\cdot \nabla ]v}_{L^2}\lesssim d_j 2^{-js}\norm{\nabla u}_{\dot{B}_{2,1}^{\frac{n}{2}}}\norm{v}_{\dot{B}_{2,1}^{s}}.
$$
\end{lemma}
\begin{lemma}\label{fuhe}{\rm(\cite{bcd})}
   Let $G$ with $G(0)=0$ be a smooth function defined on an open interval $I$
of $\R$ containing~$0.$
Then  the following estimates
$$
\norm{G(f)}_{\dot B^{s}_{2,1}}\lesssim\norm{f}_{\dot B^s_{2,1}}\quad\hbox{and}\quad
\norm{G(f)}_{\widetilde{L}^q_T(\dot B^{s}_{2,1})}\lesssim\norm{f}_{\widetilde{L}^q_T(\dot B^s_{2,1})}
$$
hold true for  $s>0,$ $1\leq  q\leq \infty$ and
 $f$  valued in a bounded interval $J\subset I.$
\end{lemma}
\begin{lemma} [\cite{bcd}]\label{heat}
Let $\sigma\in \R$,  $T>0$,  and $1\leq q_{2}\leq q_{1}\leq\infty$.  Let $u$  satisfy the heat equation
$$\partial_tu-\Delta u=f,\quad
u_{|t=0}=u_0.$$
Then  there holds the following a priori estimate
\begin{align*}
\norm{u}_{\widetilde L_{T}^{q_1}(\dot B^{\sigma+\frac 2{q_1}}_{2,1})}\lesssim
\norm{u_0}_{\dot B^\sigma_{2,1}}+\norm{f}_{\widetilde L^{q_2}_{T}(\dot B^{\sigma-2+\frac 2{q_2}}_{2,1})}.
\end{align*}
\end{lemma}

\medskip
\section{Proof of  Theorem \ref{dingli2}}
In this section, we  complete the proof of Theorem \ref{dingli2} by the following four subsections.
\subsection{  Estimates  for incompressible part of the velocity $\p\u$}
First, we apply the operator $\p$ to the second equation of \eqref{mm3} to get
\begin{align*}
\partial_t {\p \u}-\Delta {\p \u}=-{\p (\u\cdot \nabla \u)}+\p (I(a)\nabla{b}-I(a)\A \u) .
\end{align*}
By Lemma \ref{heat}, there holds
\begin{align}\label{han12}
\norm{{\p \u}}_{\widetilde{L}^{\infty}_t(\dot{B}^{\frac{n}{2}-1}_{2,1})}
+\norm{{\p \u}}_{L^1_t(\dot{B}^{\frac{n}{2}+1}_{2,1})}
\lesssim&\norm{\p \u_0}_{\dot{B}^{\frac{n}{2}-1}_{2,1}}
+\norm{\u\cdot\nabla \u}_{L^1_t(\dot{B}^{\frac{n}{2}-1}_{2,1})}\nn\\
&+\norm{I(a)\nabla{b}}_{L^1_t(\dot{B}^{\frac{n}{2}-1}_{2,1})}+\norm{I(a)\A \u}_{L^1_t(\dot{B}^{\frac{n}{2}-1}_{2,1})}.
\end{align}
\subsection{ Estimates for the low frequency part of $(a,b,\q u)$}
We cannot use  directly the equation of $a$ to obtain
$\norm{a^{\ell}}_{\widetilde{L}^{\infty}_t(\dot{B}^{\frac{n}{2}-1}_{2,1})}$ as there is no control
$\norm{\div \u^{\ell}}_{L^1_t(\dot{B}^{\frac{n}{2}-1}_{2,1})}$. Here, we introduce an unknown good function
$$\phi\stackrel{\mathrm{def}}{=}\gamma a-b$$
to overcome the difficulty.

It's straightforward to deduce from \eqref{mm3} that $\aa$ satisfies
\begin{align}\label{ni2}
&\partial_t\aa+\u\cdot\nabla \aa+\gamma (a-b)\div \u=0.
\end{align}
Applying $\ddj$ to both hand side of \eqref{ni2} and using a commutator's argument give rise to
\begin{align*}
\partial_t\ddj{\aa}+\u\cdot\nabla \ddj \aa+[\ddj,\u\cdot\nabla] \aa+\gamma\ddj( (a-b)\div \u)=0.
\end{align*}
Taking the $L^2$ inner product with $\ddj\aa$ and
multiplying by $1/\| \dot{\Delta}_j\aa\|_{L^2}2^{j(\frac{n}{2}-1)}$  formally on both hand side,  then integrating  the resultant inequality from $0$ to $t$, we can get by summing up about $j\le j_0$ that
\begin{align}\label{ming2-1}
\norm{\aa^{\ell}}_{\widetilde{L}^{\infty}_t(\dot{B}^{\frac{n}{2}-1}_{2,1})}
\lesssim&\norm{\aa^{\ell}_0}_{\dot{B}^{\frac{n}{2}-1}_{2,1}}
+\int_0^t\|\div \u\|_{L^\infty}\|\aa^\ell\|_{\dot B^{\frac  n2-1}_{2,1}}\,ds\nn\\
&+\int_0^t\sum_{j\le j_0}2^{(\frac {n}{2}-1)j}\big\|[\ddj,\u\cdot\nabla]\aa\big\|_{L^2}\,ds
+\gamma\int_0^t\|( (a-b)\div \u)^\ell\|_{\dot B^{\frac  n2-1}_{2,1}}\,ds.
\end{align}
Thanks to Lemmas \ref{law}, \ref{jiaohuanzi}  and  the embedding relation ${\dot B^{\frac  n2}_{2,1}(\R^n)}\hookrightarrow L^\infty(\R^n)$, there holds
\begin{align}\label{}
&\|\div \u\|_{L^\infty}\|\aa^\ell\|_{\dot B^{\frac  n2-1}_{2,1}}+\sum_{j\le j_0}2^{(\frac {n}{2}-1)j}
\big\|[\ddj,\u\cdot\nabla]\aa\big\|_{L^2}+\gamma
\|( (a-b)\div \u)^\ell\|_{\dot B^{\frac  n2-1}_{2,1}}\nn\\
&\quad\lesssim(\|\aa\|_{\dot B^{\frac  n2-1}_{2,1}}+\gamma\| a-b\|_{\dot B^{\frac  n2-1}_{2,1}})\norm{\nabla\u}_{\dot{B}^{\frac{n}{2}}_{2,1}},
\end{align}
from which we can  get
\begin{align}\label{ming2}
\norm{\aa^{\ell}}_{\widetilde{L}^{\infty}_t(\dot{B}^{\frac{n}{2}-1}_{2,1})}
\lesssim&\norm{\aa^{\ell}_0}_{\dot{B}^{\frac{n}{2}-1}_{2,1}}
+\int_0^t(\norm{(a^{\ell},b^{\ell})}_{\dot{B}^{\frac{n}{2}-1}_{2,1}}+\norm{( a^{h},b^{h})}_{\dot{B}^{\frac{n}{2}}_{2,1}})
\norm{\u}_{\dot{B}^{\frac{n}{2}+1}_{2,1}}\,ds.
\end{align}
\vskip .1in
For studying  the coupling between $b$ and $\q\u,$ it
is convenient to set  $\v=\Lambda^{-1}\div \u$ (with $\Lambda^sz\stackrel{\mathrm{def}}{=} \mathcal{F}^{-1}(|\xi|^s\mathcal{F}z)(s\in\R)$),
keeping in mind  that,   bounding $\v$ or $\q\u$  is equivalent,
as one can go from $\v$ to $\q\u$ or from $\q\u$ to $\v$ by means of
a $0$ order homogeneous Fourier multiplier.
Now,  one can  infer from \eqref{mm3} that
\begin{eqnarray}\label{ping2}
\left\{\begin{aligned}
&\partial_tb+\gamma\div \u+\u\cdot\nabla b+\gamma b\div \u=0,\\
&\partial_t\v + \Lambda^{-1}\div(\u\cdot \nabla \u) - 2\Delta \v -\Lambda b =\Lambda^{-1}\div\mathbf{F}(a,\u,{b}),
\end{aligned}\right.
\end{eqnarray}
with
$
\mathbf{F}(a,\u,{b})\stackrel{\mathrm{def}}{=}I(a)\nabla{b}-I(a)\A \u.
$

The estimates on the dissipation of  $b,\q\u$ in the low frequency part  are presented in the following lemma.
\begin{lemma}\label{ping6}
Under the assumption of Theorem \ref{dingli2}, we have
\begin{align}\label{ping151213}
&\norm{(b^{\ell},\q\u^{\ell})}_{\widetilde{L}^{\infty}_t(\dot{B}^{\frac{n}{2}-1}_{2,1})}
+\norm{(b^{\ell},\q \u^{\ell})}_{L^1_t(\dot{B}^{\frac{n}{2}+1}_{2,1})}\notag\\
&\quad\lesssim\norm{(b^{\ell}_0,\u^{\ell}_0)}_{\dot{B}^{\frac{n}{2}-1}_{2,1}}
+\norm{\u\cdot\nabla b}_{L^1_t(\dot{B}^{\frac{n}{2}-1}_{2,1})}
\nn\\
&\qquad+\gamma\norm{b\div\u}_{L^1_t(\dot{B}^{\frac{n}{2}-1}_{2,1})}+\norm{\u\cdot\nabla \u}_{L^1_t(\dot{B}^{\frac{n}{2}-1}_{2,1})}+\norm{\mathbf{F}(a,\u,{b})}_{L^1_t(\dot{B}^{\frac{n}{2}-1}_{2,1})}.
\end{align}
\end{lemma}
\begin{proof}
The linear equation of $b,\v$ in \eqref{ping2} coincides with the compressible Navier-Stokes equations, hence, we can follow the method used in \cite{danchin2000} to get the desired estimates. Here, we sketch its proof for readers' convenience.
We first rewrite \eqref{ping2}  into
\begin{eqnarray}\label{ping4}
\left\{\begin{aligned}
&\partial_tb+\gamma\Lambda \v=f_1, \\
&\partial_t\v -2\Delta \v -\Lambda {b}= f_2,
\end{aligned}\right.
\end{eqnarray}
with
\begin{align}\label{shaah}
f_1\stackrel{\mathrm{def}}{=}&-\u\cdot\nabla b-\gamma b\div\u,\quad
f_2\stackrel{\mathrm{def}}{=}\Lambda^{-1}\div(-\u\cdot\nabla \u+\mathbf{F}(a,\u,{b})).
\end{align}

Let $k_0$ be some integer, and $z^\ell\stackrel{\mathrm{def}}{=}\dot{S}_{k_0}z$.
Denote
$f_k=\dot{\Delta}_kf$, applying the operator $\dot{\Delta}_kS_{k_0}$ to the  equations in \eqref{ping4}, then multiplying
$(\ref{ping4})_1$ by $b_{k}^{\ell}$, $(\ref{ping4})_2$ by $\gamma\v_{k}^{\ell}$, respectively, we can get
\begin{align}\label{ping8}
&\frac12\frac{d}{dt}(\norm{b_{k}^{\ell}}^2_{L^2}+\gamma\norm{\v_k^{\ell}}^2_{L^2})+2\gamma\norm{\Lambda \v_k^{\ell}}^2_{L^2}=\big\langle(f_{1})^{\ell}_k,b^{\ell}_{k}\big\rangle+\big\langle(f_{2})^{\ell}_k ,\gamma\v_k^{\ell}\big\rangle.
\end{align}
To capture the dissipation of $b$, we have to consider the time derivative of the mixed terms involved in $\big\langle\v_k^{\ell},\Lambda b_{k}^{\ell}\big\rangle:$
\begin{align}\label{ping10}
&-\frac12\frac{d}{dt}\big\langle\v_k^{\ell},\Lambda {b}_{k}^{\ell}\big\rangle+\frac12\norm{\Lambda {b}_{k}^{\ell}}^2_{L^2}-\frac\gamma2\norm{\Lambda \v_k^{\ell}}^2_{L^2}=-\big\langle\Delta \v_k^{\ell}, \Lambda {b}_{k}^{\ell}\big\rangle-\frac12\big\langle(\Lambda f_{1})^{\ell}_k , {\v}_{k}^{\ell}\big\rangle-\frac12\big\langle(f_{2})^{\ell}_k ,\Lambda {b}_{k}^{\ell}\big\rangle.
\end{align}
To eliminate the highest order term appeared on the right hand side of \eqref{ping10},
we next want an estimate for  $\norm{\Lambda b_{k}^{\ell}}^2_{L^2}$.
 From \eqref{ping4}, we have
\begin{align*}
\partial_t\Lambda b_{k}^{\ell}+\gamma\Lambda^2 \v_{k}^{\ell}=\Lambda(f_1)_{k}^{\ell}.
\end{align*}
Testing that equation  by $\Lambda b_{k}^{\ell}$ yields
\begin{align}\label{ping12}
\frac{1}{2\gamma}\frac{d}{dt}\norm{\Lambda {b}_{k}^{\ell}}^2_{L^2}
=\big\langle\Delta \v_k^{\ell}, \Lambda {b}_{k}^{\ell}\big\rangle+\frac{1}{\gamma}\big\langle(f_{1})^{\ell}_k ,\Lambda^2 {b}_{k}^{\ell}\big\rangle.
\end{align}
Denote
$$\mathcal{L}^2_k\stackrel{\mathrm{def}}{=}\norm{b_{k}^{\ell}}^2_{L^2}+\gamma\norm{\v_k^{\ell}}^2_{L^2}
+\frac{1}{\gamma}\norm{\Lambda{b}_{k}^{\ell}}^2_{L^2}
-\big\langle\v_k^{\ell},\Lambda {b}_{k}^{\ell}\big\rangle.$$

Summing up $\eqref{ping8}$, \eqref{ping10}, and $\eqref{ping12}$,  we obtain
\begin{align*}
\frac12\frac{d}{dt}\mathcal{L}^2_k+\frac{3\gamma}{2}\norm{\Lambda \v_k^{\ell}}^2_{L^2}+\frac{1}{2}\norm{\Lambda {b}_{k}^{\ell}}^2_{L^2}=&
\big\langle(f_{1})^{\ell}_k,b^{\ell}_{k}\big\rangle+
\big\langle(f_{2})^{\ell}_k ,\gamma\v_k^{\ell}\big\rangle
\nn\\
&-\frac12\big\langle(\Lambda f_{1})^{\ell}_k , {\v}_{k}^{\ell}\big\rangle
-\frac12\big\langle(f_{2})^{\ell}_k ,\Lambda {b}_{k}^{\ell}\big\rangle+\frac{1}{\gamma}\big\langle(f_{1})^{\ell}_k ,\Lambda^2 {b}_{k}^{\ell}\big\rangle.
\end{align*}
It's straightforward to  deduce from the low-frequency cut-off and Young inequality that
\begin{align*}
\mathcal{L}^2_k\thickapprox\Big\|{(b_{k}^{\ell}, \gamma\v^{\ell}_k ,\frac{1}{\gamma}\Lambda b^{\ell}_k)}\Big\|^2_{L^2}\thickapprox\norm{(b_{k}^{\ell}, \gamma\v^{\ell}_k)}^2_{L^2},
\end{align*}
which leads to
\begin{align}\label{ping14}
\frac12\frac{d}{dt}\mathcal{L}^2_k+2^{2k}\mathcal{L}^2_k\lesssim\norm{((f_{1})^{\ell}_k ,(f_{2})^{\ell}_k )}_{L^2}\mathcal{L}_k.
\end{align}
Multiplying the above inequality by $2^{(\frac{n}{2}-1)j}/\mathcal{L}_k$ formally on both hand sides, and then integrating from $0$ to $t$, summing up about $j\leq j_0$, we finally get that
\begin{align}\label{ping7}
&\norm{(b^{\ell},\gamma\v^{\ell})}_{\widetilde{L}^{\infty}_t(\dot{B}^{\frac{n}{2}-1}_{2,1})}
+\norm{(b^{\ell},\gamma\v^{\ell})}_{L^1_t(\dot{B}^{\frac{n}{2}+1}_{2,1})}\lesssim\norm{(b^{\ell}_0,\gamma\v^{\ell}_0)}_{\dot{B}^{\frac{n}{2}-1}_{2,1}}
+\int_0^t\big\|{\left((f_{1})^{\ell},(f_{2})^{\ell}\right)\big\|}_{\dot{B}^{\frac{n}{2}-1}_{2,1}}\,ds.
\end{align}

The combination of \eqref{ping7} and \eqref{shaah} implies \eqref{ping151213}.
This proves the lemma.
\end{proof}

\subsection{ Estimates for the high frequency part of $(a, b,\q u)$}
We first present the estimate of the high frequency part of $a$. Taking  similar processes as the derivation of \eqref{ming2}, one can infer from the first equation of \eqref{mm3} that
\begin{align}\label{ming4}
\norm{a^{h}}_{\widetilde{L}^{\infty}_t(\dot{B}^{\frac{n}{2}}_{2,1})}
\lesssim&\norm{a^{h}_0}_{\dot{B}^{\frac{n}{2}}_{2,1}}+\int_0^t
\norm{\u}_{\dot{B}^{\frac{n}{2}+1}_{2,1}}\,ds
+\int_0^t\big(\norm{a^{\ell}}_{\dot{B}^{\frac{n}{2}-1}_{2,1}}+\norm{a^{h}}_{\dot{B}^{\frac{n}{2}}_{2,1}}\big)
\norm{\u}_{\dot{B}^{\frac{n}{2}+1}_{2,1}}\,ds.
\end{align}

We next deal with the high frequency estimates of $b,\q\u.$ By using the operator $\q$, we infer from \eqref{mm3} that
\begin{eqnarray}\label{ping18}
\left\{\begin{aligned}
&\partial_t{b}+\gamma\div \u=-\u\cdot\nabla b-\gamma b\div\u,\\
&\partial_t\mathcal{Q}\u - 2\Delta \mathcal{Q}\u +\nabla {b} =-\mathcal{Q}(\u\cdot\nabla \u)+\mathcal{Q}\mathbf{F}(a,\u,{b}).
\end{aligned}\right.
\end{eqnarray}
Subsequently, we perform the energy argument in terms
of the effective velocity by following the approach used in \cite{zhaixiaopingjmfm2019}, \cite{helingbing}, and \cite{haspot} that
\begin{align}\label{han22}
{\mathbf{G}}\stackrel{\mathrm{def}}{=} \mathcal{Q}\u-\frac12\Delta^{-1}\nabla {b},
\end{align}
then ${\mathbf{G}}$ satisfies
\begin{align*}
\partial_t{\mathbf{G}}-2\Delta {\mathbf{G}}=\frac\gamma2{\mathbf{G}}+\frac\gamma4\Delta^{-1}\nabla {b}+\frac12\mathcal{Q}({b} \u)+\frac{\gamma-1}{2}\Delta^{-1}\nabla (b\div\u)-\mathcal{Q}(\u\cdot\nabla \u)+\mathcal{Q}\mathbf{F}(a,\u,{b}),
\end{align*}
By taking $\sigma=\frac n2-1,$ $q_1=\infty,\hbox{or}\  q_1=1$, and $q_2=1$ in Lemma \ref{heat} respectively, we can get the estimate of  $\mathbf{G}$  in the high frequencies as follows
\begin{align}\label{han1}
&\norm{\mathbf{G}^h}_{ \widetilde{L}_t^\infty(\dot B^{\frac  n2-1}_{2,1})}+2 \norm{\mathbf{G}^h}_{L^1_t(\dot B^{\frac  n2+1}_{2,1})}\nn\\
&\quad\lesssim \norm{\mathbf{G}_0^h}_{\dot B^{\frac  n2-1}_{2,1}}
+ \int_0^t\norm{\mathbf{G}^h}_{\dot B^{\frac  n2-1}_{2,1}}\,ds+ \int_0^t\norm{\Delta^{-1}\nabla b^h}_{\dot B^{\frac  n2-1}_{2,1}}\,ds+\int_0^t \norm{\mathcal{Q}({b} \u)^h}_{\dot B^{\frac  n2-1}_{2,1}}\,ds\nonumber\\
&\qquad+\int_0^t \norm{(\Delta^{-1}\nabla({b} \div\u))^h}_{\dot B^{\frac  n2-1}_{2,1}}\,ds+\int_0^t\norm{\u\cdot\nabla \u}_{\dot B^{\frac  n2-1}_{2,1}}\,ds+\int_0^t\norm{\mathbf{F}(a,\u,{b})}_{\dot B^{\frac  n2-1}_{2,1}}\,ds.
\end{align}
The important point is that, owing to the high frequency cut-off at $|\xi|\sim 2^{j_0},$
$$
\norm{\mathbf{G}^h}_{L^1_t(\dot B^{\frac  n2-1}_{2,1})}\lesssim 2^{-2j_0}\norm{\mathbf{G}^h}_{L^1_t(\dot B^{\frac  n2+1}_{2,1})}
\quad\hbox{and}\quad
\norm{b^h}_{L^1_t(\dot B^{\frac  n2-2}_{2,1})}\lesssim  2^{-2j_0}\norm{b^h}_{L^1_t(\dot B^{\frac  n2}_{2,1})}.
$$
Hence, if $j_0$ is large enough then the term $\norm{\mathbf{G}^h}_{L^1_t(\dot B^{\frac  n2-1}_{2,1})}$ may be absorbed in the right hand side.

In view of \eqref{han22}, one can  get
the equation of  ${b}$
\begin{align}\label{han2}
\partial_t{b}+\frac{\gamma}{2}{b}+\u\cdot\nabla b=-\gamma\div {\mathbf{G}}-\gamma{b}\div \u.
\end{align}
Applying $\ddj$ to both hand side of \eqref{han2} and using a commutator's argument give rise to
\begin{align*}
\partial_t\ddj{b}+\frac\gamma2\ddj{b}+\u\cdot\nabla \ddj b=-\gamma\ddj\div {\mathbf{G}}-[\ddj,\u\cdot\nabla] b-\ddj(\gamma{b}\div \u).
\end{align*}
A standard energy argument leads to
\begin{align*}
&\|\ddj b(t)\|_{L^2}+\frac\gamma2\int_0^t\|\ddj b\|_{L^2}\,ds\nonumber\\
&\quad\lesssim\|\ddj b_0\|_{L^2}+\int_0^t\|\ddj\div \mathbf{G}\|_{L^2}\,ds
\nonumber\\
&\quad\quad+\int_0^t\norm{\div \u}_{L^\infty}\|\ddj b\|_{L^2}\,ds+\int_0^t\big\|[\ddj,\u\cdot\nabla] b\big\|_{L^2}\,ds+\int_0^t\|\ddj(b\div  \u)\|_{L^2}\,ds
\end{align*}
from which and Lemma \ref{jiaohuanzi} we can further get
\begin{align}\label{han6}
\norm{b^h}_{\widetilde{L}^\infty_t(\dot B^{\frac  n2}_{2,1})}+\frac\gamma2\norm{b^h}_{ L^1_t(\dot B^{\frac  n2}_{2,1})}\lesssim& \norm{b_0^h}_{\dot B^{\frac  n2}_{2,1}}+\int_0^t\norm{\mathbf{G}^h}_{\dot B^{\frac  n2+1}_{2,1}}\,ds+\int_0^t\norm{ \nabla\u}_{\dot B^{\frac  n2}_{2,1}}\norm{b}_{\dot B^{\frac  n2}_{2,1}}\,ds.
\end{align}
Multiplying by a suitable large constant on both hand side of \eqref{han1} and then summing up the resultant and \eqref{han6}, we can infer that
\begin{align}\label{han7}
&\norm{{\mathbf{G}}^h}_{\widetilde{L}^{\infty}_t(\dot{B}^{\frac{n}{2}-1}_{2,1})}+\norm{{b}^h}_{\widetilde{L}^{\infty}_t(\dot{B}^{\frac{n}{2}}_{2,1})}
+\norm{{\mathbf{G}}^h}_{{L}^{1}_t(\dot{B}^{\frac{n}{2}+1}_{2,1})}+\norm{{b}^h}_{{L}^{1}_t(\dot{B}^{\frac{n}{2}}_{2,1})}\notag\\
&\quad\lesssim \norm{{\mathbf{G}}^h_{0}}_{\dot{B}^{\frac{n}{2}-1}_{2,1}}+\norm{{b}^h_{0}}_{\dot{B}^{\frac{n}{2}}_{2,1}}
+\int_0^t\norm{\nabla \u}_{\dot B^{\frac  n2}_{2,1}}\norm{b}_{\dot B^{\frac  n2}_{2,1}}\,ds+\int^t_0\norm{({b} \div\u)^h}_{\dot{B}^{\frac{n}{2}-2}_{2,1}}\,ds\notag\\
&\quad\quad +\int^t_0(
\norm{\u\cdot\nabla \u}_{\dot{B}^{\frac{n}{2}-1}_{2,1}}+\norm{({b} \u)^h}_{\dot{B}^{\frac{n}{2}-1}_{2,1}})\,ds
+\int^t_0\norm{\mathbf{F}(a,\u,{b})}_{\dot{B}^{\frac{n}{2}-1}_{2,1}}\,ds.
\end{align}
In view of ${\mathbf{G}}\stackrel{\mathrm{def}}{=} \mathcal{Q}\u-\frac12\Delta^{-1}\nabla b$ and the embedding relation in the high frequency, there hold
\begin{align*}
\norm{\q \u^h}_{\widetilde{L}^{\infty}_t(\dot{B}^{\frac{n}{2}-1}_{2,1})}
\lesssim& \norm{\mathbf{G}^h}_{\widetilde{L}^{\infty}_t(\dot{B}^{\frac{n}{2}-1}_{2,1})}
+\norm{ b^h}_{\widetilde{L}^{\infty}_t(\dot{B}^{\frac{n}{2}}_{2,1})},
\end{align*}
\begin{align*}
\norm{\q \u^h}_{{L}^{1}_t(\dot{B}^{\frac{n}{2}+1}_{2,1})}
\lesssim & \norm{\mathbf{G}^h}_{{L}^{1}_t(\dot{B}^{\frac{n}{2}+1}_{2,1})}
+\norm{ b^h}_{{L}^{1}_t(\dot{B}^{\frac{n}{2}}_{2,1})}.
\end{align*}
As a result, we can rewrite \eqref{han7} into
\begin{align}\label{han10}
&\norm{b^h}_{\widetilde{L}^{\infty}_t(\dot{B}^{\frac{n}{2}}_{2,1})}
+\norm{\q\u^h}_{\widetilde{L}^{\infty}_t(\dot{B}^{\frac{n}{2}-1}_{2,1})}+\norm{b^h}_{{L}^{1}_t(\dot{B}^{\frac{n}{2}}_{2,1})}
+\norm{{\q\u}^h}_{{L}^{1}_t(\dot{B}^{\frac{n}{2}+1}_{2,1})}\nn\\
&\quad\lesssim \norm{{b}^h_{0}}_{\dot{B}^{\frac{n}{2}}_{2,1}}+
\norm{{\q\u}^h_{0}}_{\dot{B}^{\frac{n}{2}-1}_{2,1}}+\int_0^t\norm{ \nabla\u}_{\dot B^{\frac  n2}_{2,1}}\norm{b}_{\dot B^{\frac  n2}_{2,1}}\,ds
+\int^t_0\norm{({b} \div\u)^h}_{\dot{B}^{\frac{n}{2}-2}_{2,1}}\,ds\nn\\
&\qquad +\int^t_0(
\norm{\u\cdot\nabla \u}_{\dot{B}^{\frac{n}{2}-1}_{2,1}}+\norm{({b} \u)^h}_{\dot{B}^{\frac{n}{2}-1}_{2,1}}+\norm{\mathbf{F}(a,\u,{b})}_{\dot{B}^{\frac{n}{2}-1}_{2,1}})\,ds.
\end{align}

Combining   \eqref{ming2}, \eqref{ping151213}, \eqref{ming4}, with \eqref{han10} gives
\begin{align}\label{nihao}
&\norm{(\aa^{\ell},b^{\ell}, \u)}_{\widetilde{L}^{\infty}_t(\dot{B}^{\frac{n}{2}-1}_{2,1})}+\norm{(a^h,b^h)}_{\widetilde{L}^{\infty}_t(\dot{B}^{\frac{n}{2}}_{2,1})}
+\norm{(b^{\ell}, \u)}_{L^1_t(\dot{B}^{\frac{n}{2}+1}_{2,1})}+\norm{b^h}_{{L}^{1}_t(\dot{B}^{\frac{n}{2}}_{2,1})}
\nn\\
&\quad\lesssim \norm{(a_0^\ell,b_0^\ell,{\u}_{0})}_{\dot{B}^{\frac{n}{2}-1}_{2,1}}+
\norm{({a}^h_{0},{b}^h_{0})}_{\dot{B}^{\frac{n}{2}}_{2,1}}+\int_0^t\norm{ \nabla\u}_{\dot B^{\frac  n2}_{2,1}}\norm{b}_{\dot B^{\frac  n2}_{2,1}}\,ds+\int^t_0\norm{({b} \div\u)^h}_{\dot{B}^{\frac{n}{2}-2}_{2,1}}\,ds
\nn\\
&\qquad
+\int_0^t(\norm{(a^{\ell},b^{\ell})}_{\dot{B}^{\frac{n}{2}-1}_{2,1}}+\norm{( a,b)^{h}}_{\dot{B}^{\frac{n}{2}}_{2,1}})
\norm{\u}_{\dot{B}^{\frac{n}{2}+1}_{2,1}}\,ds+\int^t_0(
\norm{\u\cdot\nabla \u}_{\dot{B}^{\frac{n}{2}-1}_{2,1}}+\norm{({b} \u)^h}_{\dot{B}^{\frac{n}{2}-1}_{2,1}}
)\,ds
\nn\\
&\qquad +\int^t_0(\norm{\u\cdot\nabla b}_{\dot{B}^{\frac{n}{2}-1}_{2,1}}+\gamma\norm{b\div\u}_{\dot{B}^{\frac{n}{2}-1}_{2,1}}
+\norm{\mathbf{F}(a,\u,{b})}_{\dot{B}^{\frac{n}{2}-1}_{2,1}})\,ds.
\end{align}

We now bound each terms on the right hand side of \eqref{nihao}.
First, according to product law in Lemma \ref{law}, we have
\begin{align}\label{han14}
\norm{ \nabla\u}_{\dot B^{\frac  n2}_{2,1}}\norm{b}_{\dot B^{\frac  n2}_{2,1}}+\norm{({b} \div\u)^h}_{\dot{B}^{\frac{n}{2}-1}_{2,1}}
\lesssim&\norm{ \nabla\u}_{\dot B^{\frac  n2}_{2,1}}\norm{b}_{\dot B^{\frac  n2}_{2,1}}+\norm{{b} \div\u}_{\dot{B}^{\frac{n}{2}}_{2,1}}\nn\\
\lesssim&(\norm{b^\ell}_{\dot{B}^{\frac{n}{2}-1}_{2,1}}+\norm{b^h}_{\dot{B}^{\frac{n}{2}}_{2,1}})
\norm{\u}_{\dot{B}^{\frac{n}{2}+1}_{2,1}},\nn\\
\norm{\u\cdot\nabla \u}_{\dot{B}^{\frac{n}{2}-1}_{2,1}}
\lesssim&\norm{\u}_{\dot{B}^{\frac{n}{2}-1}_{2,1}}\norm{\u}_{\dot{B}^{\frac{n}{2}+1}_{2,1}}.
\end{align}
By Lemma \ref{law}, the embedding relation in high frequency,  the Young inequality, and the interpolation inequality, there holds
\begin{align}\label{han16}
&\norm{\u\cdot\nabla b}_{\dot{B}^{\frac{n}{2}-1}_{2,1}}+\gamma\norm{b\div\u}_{\dot{B}^{\frac{n}{2}-1}_{2,1}}+\norm{(b \u)^h}_{\dot{B}^{\frac{n}{2}-1}_{2,1}}\nn\\
&\quad\lesssim\norm{\u}_{\dot{B}^{\frac{n}{2}}_{2,1}}\norm{\nabla b}_{\dot{B}^{\frac{n}{2}-1}_{2,1}}+\gamma\norm{b}_{\dot{B}^{\frac{n}{2}}_{2,1}}\norm{\div \u}_{\dot{B}^{\frac{n}{2}-1}_{2,1}}+\norm{b \u}_{\dot{B}^{\frac{n}{2}}_{2,1}}\nn\\
&\quad\lesssim\norm{b }_{\dot{B}^{\frac{n}{2}}_{2,1}}^2+\norm{ \u}_{\dot{B}^{\frac{n}{2}}_{2,1}}^2\nn\\
&\quad\lesssim\norm{b^\ell}_{\dot{B}^{\frac{n}{2}-1}_{2,1}}\norm{b^\ell}_{\dot{B}^{\frac{n}{2}+1}_{2,1}}+\norm{b^h }_{\dot{B}^{\frac{n}{2}}_{2,1}}^2
+\norm{\u}_{\dot{B}^{\frac{n}{2}-1}_{2,1}}\norm{\u}_{\dot{B}^{\frac{n}{2}+1}_{2,1}}.
\end{align}
At last, we deal with each terms in $\mathbf{F}(a,\u,{b})$.
We first use the fact that
$I(a)=a-aI(a)$ and  Lemmas \ref{law}, \ref{fuhe} to get
\begin{align*}
 \norm{I(a)}_{\dot{B}_{2,1}^{\frac n2-1}}
\lesssim&\norm{a}_{\dot{B}_{2,1}^{\frac n2-1}}+\norm{a}_{\dot{B}_{2,1}^{\frac n2-1}}\norm{I(a)}_{\dot{B}_{2,1}^{\frac n2}}
 \nonumber\\
 \lesssim&(\norm{a^\ell}_{\dot{B}_{2,1}^{\frac n2-1}}+\norm{a^h}_{\dot{B}_{2,1}^{\frac n2}})+(\norm{a^\ell}_{\dot{B}_{2,1}^{\frac n2-1}}+\norm{a^h}_{\dot{B}_{2,1}^{\frac n2}})\norm{a}_{\dot{B}_{2,1}^{\frac n2}}
 \nonumber\\
 \lesssim&(\norm{a^\ell}_{\dot{B}_{2,1}^{\frac n2-1}}+\norm{a^h}_{\dot{B}_{2,1}^{\frac n2}}+1)(\norm{a^\ell}_{\dot{B}_{2,1}^{\frac n2-1}}+\norm{a^h}_{\dot{B}_{2,1}^{\frac n2}}).
\end{align*}
Hence, in view of $I(a) {\nabla b}=I(a) {\nabla b}^\ell+I(a) {\nabla b}^h$, we further obtain
\begin{align}\label{han18}
\norm{I(a) {\nabla b}}_{\dot{B}_{2,1}^{\frac n2-1}}
\lesssim&
\norm{I(a)}_{\dot{B}_{2,1}^{\frac n2-1}}\norm{{\nabla b}^\ell}_{\dot{B}_{2,1}^{\frac n2}}+\norm{I(a)}_{\dot{B}_{2,1}^{\frac n2}}\norm{{\nabla b}^h}_{\dot{B}_{2,1}^{\frac n2-1}}\nonumber\\
\lesssim&
\norm{I(a)}_{\dot{B}_{2,1}^{\frac n2-1}}\norm{{\nabla b}^\ell}_{\dot{B}_{2,1}^{\frac n2}}+\norm{a}_{\dot{B}_{2,1}^{\frac n2}}\norm{{\nabla b}^h}_{\dot{B}_{2,1}^{\frac n2-1}}\nonumber\\
\lesssim
&(\norm{a^\ell}_{\dot{B}_{2,1}^{\frac n2-1}}+\norm{a^h}_{\dot{B}_{2,1}^{\frac n2}})\norm{{ b}^h}_{\dot{B}_{2,1}^{\frac n2}}\nonumber\\
&+ (\norm{a^\ell}_{\dot{B}_{2,1}^{\frac n2-1}}+\norm{a^h}_{\dot{B}_{2,1}^{\frac n2}}+1)(\norm{a^\ell}_{\dot{B}_{2,1}^{\frac n2-1}}+\norm{a^h}_{\dot{B}_{2,1}^{\frac n2}})\norm{{ b}^\ell}_{\dot{B}_{2,1}^{\frac n2+1}}.
\end{align}
Similarly,
\begin{align}\label{han19}
\norm{(I(a)(\Delta \u+\nabla\div \u))^h}_{\dot{B}^{\frac{n}{2}-1}_{2,1}}
\lesssim&(\norm{a^\ell}_{\dot{B}_{2,1}^{\frac n2-1}}+\norm{a^h}_{\dot{B}_{2,1}^{\frac n2}})\norm{\u}_{\dot{B}^{\frac{n}{2}+1}_{2,1}}.
\end{align}
From the definition of $\aa$, we have
 $$\norm{a^\ell}_{\dot{B}_{2,1}^{\frac n2-1}}\lesssim\norm{\aa^\ell}_{\dot{B}_{2,1}^{\frac n2-1}}+\norm{b^\ell}_{\dot{B}_{2,1}^{\frac n2-1}}.$$

Hence, collecting the above estimates \eqref{han14}--\eqref{han19}, we can finally get from \eqref{nihao} that
\begin{align}\label{ming5}
&\norm{(a^{\ell},b^{\ell}, \u)}_{\widetilde{L}^{\infty}_t(\dot{B}^{\frac{n}{2}-1}_{2,1})}+\norm{(a^h,b^h)}_{\widetilde{L}^{\infty}_t(\dot{B}^{\frac{n}{2}}_{2,1})}
+\norm{(b^{\ell}, \u)}_{L^1_t(\dot{B}^{\frac{n}{2}+1}_{2,1})}+\norm{b^h}_{{L}^{1}_t(\dot{B}^{\frac{n}{2}}_{2,1})}
\nn\\
&\quad\lesssim \norm{(a_0^\ell,b_0^\ell,{\u}_{0})}_{\dot{B}^{\frac{n}{2}-1}_{2,1}}+
\norm{({a}^h_{0},{b}^h_{0})}_{\dot{B}^{\frac{n}{2}}_{2,1}}\nn\\
&\qquad+\int_0^t(\norm{(a^\ell,b^\ell)}_{\dot{B}_{2,1}^{\frac n2-1}}+\norm{a^h}_{\dot{B}_{2,1}^{\frac n2}})(\norm{b^h}_{\dot{B}^{\frac{n}{2}}_{2,1}}+\norm{\u}_{\dot{B}^{\frac{n}{2}+1}_{2,1}})\,ds\nn\\
&\qquad+\int_0^t(\norm{(b^\ell,\u)}_{\dot{B}^{\frac{n}{2}-1}_{2,1}}+\norm{b^h}_{\dot{B}^{\frac{n}{2}}_{2,1}})
(\norm{(b^\ell,\u)}_{\dot{B}^{\frac{n}{2}+1}_{2,1}}+\norm{b^h}_{\dot{B}^{\frac{n}{2}}_{2,1}})\,ds\nn\\
&\qquad+\int_0^t(\norm{(a^\ell,b^\ell)}_{\dot{B}_{2,1}^{\frac n2-1}}+\norm{a^h}_{\dot{B}_{2,1}^{\frac n2}}+1)(\norm{(a^\ell,b^\ell)}_{\dot{B}_{2,1}^{\frac n2-1}}+\norm{a^h}_{\dot{B}_{2,1}^{\frac n2}})\norm{{ b}^\ell}_{\dot{B}_{2,1}^{\frac n2+1}}\,ds.
\end{align}

\subsection{Continuity argument}
In this subsection, we complete the proof of Theorem \ref{dingli2} by the continuity arguments. Denote
\begin{align*}
&\mathcal{E}(t)\stackrel{\mathrm{def}}{=}
\norm{(a^\ell,b^{\ell}, \u)}_{\widetilde{L}^{\infty}_t(\dot{B}^{\frac{n}{2}-1}_{2,1})}+\norm{(a^h,b^h)}_{\widetilde{L}^{\infty}_t(\dot{B}^{\frac{n}{2}}_{2,1})}
+\norm{(b^{\ell}, \u)}_{L^1_t(\dot{B}^{\frac{n}{2}+1}_{2,1})}+\norm{b^h}_{{L}^{1}_t(\dot{B}^{\frac{n}{2}}_{2,1})}
\nn\\
&\mathcal{E}_0\stackrel{\mathrm{def}}{=} \norm{(a_0^\ell,b_0^\ell,{\u}_{0})}_{\dot{B}^{\frac{n}{2}-1}_{2,1}}+
\norm{({a}^h_{0},{b}^h_{0})}_{\dot{B}^{\frac{n}{2}}_{2,1}}.
\end{align*}
Subsequently, we deduce from \eqref{ming5}  that
\begin{align}\label{ming7}
\mathcal{E}(t)\leq \mathcal{E}_0+C(\mathcal{E}(t))^2(1+C\mathcal{E}(t)).
\end{align}

Under the setting of initial data in Theorem\ref{dingli2},  there exists a positive constant $C_0$ such that
$\mathcal{E}_0\leq C_0 \varepsilon$. Due to the local existence result which has been  achieved by Theorem \ref{dingli1}, there exists a positive time $T$ such that
\begin{equation}\label{ming8}
 \mathcal{E} (t) \leq 2 C_0\ \varepsilon , \quad  \forall \; t \in [0, T].
\end{equation}
Let $T^{*}$ be the largest possible time of $T$ for what \eqref{ming8} holds. Now, we only need to show $T^{*} = \infty$.  By \eqref{ming7}, we can use
 a standard continuation argument to prove that $T^{*} = \infty$ provided that $\varepsilon$ is small enough.  We omit the details here. Hence, we finish the proof of Theorem \ref{dingli2}. $\hspace{14cm}\square$

\medskip
\noindent \textbf{Acknowledgement.} This research is supported by NSFC key project under the grant number 11831003, NSFC under the grant numbers 12271179,  11971356, 11601533, and 11571118, the
Science and Technology Program of Shenzhen under the grant number 20200806104726001, the Fundamental Research Founds for the Central Universities under the grant numbers 2019MS110 and 2019MS112, the Foundation for Basic and Applied Basic Research of Guangdong
under the grant numbers 2022A1515011977 and 2022A1515012097.


\begin{thebibliography}{99}


\bibitem{bcd}
H.~Bahouri, J.Y. Chemin, R.~Danchin,
\newblock { {F}ourier {A}nalysis and {N}onlinear {P}artial {D}ifferential
  {E}quations}.
\newblock Grundlehren Math. Wiss. , vol. {\textbf{343}}, Springer-Verlag,
  Berlin, Heidelberg, 2011.
\newblock $\,$


\bibitem{charve} F. Charve, R. Danchin, { A global existence result for the compressible Navier-Stokes equations in the critical $L^p$ framework}, {\it Arch. Ration. Mech. Anal.},  {\bf 198} (2010), 233--271.


\bibitem{chenqionglei} Q. Chen, C. Miao,  Z. Zhang, { Global well-posedness for compressible Navier-Stokes
equations with highly oscillating initial velocity}, {\it Comm. Pure Appl. Math.}, {\bf 63} (2010), 1173--1224.



\bibitem{zhaixiaopingjmfm2019}
 Z.~Chen, X.~Zhai,
\newblock
{ Global large solutions and incompressible limit for the compressible Navier-Stokes equations},
\newblock  {\it J. Math. Fluid Mech.}, {\bf 21} (2019), Art. 26, 23.
\newblock $\,$

\bibitem{danchin2000} R. Danchin,
{ Global existence  in  critical spaces for compressible Navier-Stokes equations},
{\it Invent. Math.}, {\bf 141} (2000), 579--614.


\bibitem{danchin2014} R. Danchin, { A Lagrangian approach for the compressible Navier-Stokes Equations}, {\it Ann. Inst. Fourier, Grrenoble}, {\bf 64} (2014), 753--791.



\bibitem{helingbing}
R. Danchin, L. He, { The incompressible limit in $ L^p$ type critical spaces},
\newblock {\it
 Math. Ann., }{\bf 366} (2016), 1365--1402.


\bibitem{danchin2018}
R.~Danchin, P. Mucha,
{
Compressible Navier-Stokes system: large solutions and incompressible limit},
 {\it Adv. Math.}, {\bf 320} (2017), 904--925.



\bibitem{fangdaoyuan} D. Fang, T. Zhang, R. Zi,  { Global solutions to the isentropic compressible Navier-Stokes equations with a class of large initial data},   {\it SIAM J. Math. Anal.}, {\bf 50} (2018), 4983--5026.

\bibitem{Feireisl1}
E. Feireisl, {\it Dynamics of Viscous Compressible Fluids}, { Oxford University Press}, 2004.

\bibitem{Feireisl3} E. Feireisl, {On the motion of a viscous, compressible and heat conducting fluid}, {\it Indiana Univ. Math. J.}, {\bf53} (2004), 1705--1738.

\bibitem{ferisal2} E. Feireisl, P. Gwiazda, A. \'{S}wierczewska-Gwiazda,  E. Wiedemann,  { Dissipative measure-valued
solutions to the compressible Navier-Stokes system}, {\it Calc. Var. Partial Differ.},  {\bf55} (2016), 55--141.

\bibitem{fei1}
E. Feireisl, R. Klein, A. Novotn\'y, and E. Zatorska, On singular limits arising in the scale
analysis of stratified fluid flows, {\it Math. Models Methods Appl. Sci.}, {\bf26} (2016), 419--443.

\bibitem{ferisal} E. Feireisl, A. Novotn\'{y}, H. Petzeltov\'{a}, { On the global existence of globally defined
weak solutions to the Navier-Stokes equations of isentropic compressible fluids},
  {\it J. Math. Fluid Mech.},  {\bf3} (2001), 358-392.



\bibitem{Feireisl4} E. Feireisl, A. Novotn\'{y},  Y. Sun, { Suitable weak solutions to the Navier-Stokes equations
of compressible viscous fluids}, {\it Indiana Univ. Math. J.},  {\bf 60} (2011), 611--631.



\bibitem{haspot} B. Haspot, { Existence of global strong solutions in critical spaces for barotropic viscous fluids},
 {\it Arch. Ration. Mech. Anal.}, {\bf 202} (2011), 427--460.


\bibitem{huangjingchi} L. He, J. Huang, C. Wang, { Global stability of large solutions to the 3D compressible Navier-Stokes equations},
{ \it Arch. Ration. Mech. Anal.}, {\bf 234} (2019), 1167--1222.










\bibitem{huangxiangdi}
X. Huang, J. Li, Z.  Xin, { Global well-posedness of classical
solutions with large oscillations and vacuum to the
three-dimensional isentropic compressible Navier-Stokes equaitons},
 {\it Comm. Pure Appl. Math.}, {\bf 65} (2012), 549--585.

\bibitem{zhangping1}
S. Jiang, P. Zhang,  { Global spherically symmetric solutions fo the compressible isentropic Navier-Stokes equations,} {\it Comm. Math. Phys.},  {\bf 215} (2001), 559--581.

\bibitem{zhangping2}
S. Jiang, P. Zhang,  { Axisymmetric solutions of the 3D Navier-Stokes equations for compressible isentropic fluids,} {\it J. Math. Pure Appl.},  {\bf 82} (2003), 949--973.


\bibitem{klein}
R. Klein,  An applied mathematical view of meteorological modelling. In Applied mathematics
entering the 21st century, pages 227--269. SIAM, Philadelphia, PA, 2004.


\bibitem{li2019}
J. Li,  Z. Xin, { Global well-posedness and large time asymptotic behavior of
classical solutions to the compressible Navier-Stokes equations with vacuum,} {\it Ann. PDE.}, {\bf 5} (2019),  7.






\bibitem{lions}
P.-L. Lions, {\it Mathematical Topics in Fluid Mechanics: Volume 2: Compressible Models}. Oxford
Science Publication, Oxford, 1998.


\bibitem{luk1}
M. Luk\'a$\check{\mathrm{c}}$ov\'a-Medvid'ov\'a, A. Sch\"omer, Existence of dissipative solutions to the compressible
Navier-Stokes system with potential temperature transport, \newblock {\it arXiv}:2106.12435.



\bibitem{luk2}
M. Luk\'a$\check{\mathrm{c}}$ov\'a-Medvid'ov\'a, A. Sch\"omer, { DMV-strong uniqueness principle for the compressible
Navier-Stokes system with potential temperature transport}, \newblock {\it arXiv}:2106.12812.


\bibitem{mal}
D. Maltese, M. Mich\'alek, P.B. Mucha, A. Novotn\'y, M. Pokorn\'y,  E. Zatorska, Existence of
weak solutions for compressible Navier-Stokes equations with entropy transport, {\it J. Differential
Equations}, {\bf261} (2016), 4448--4485.


\bibitem{mat}  A. Matsumura, T. Nishida, { The initial value problem for the equations of motion of compressible viscous and heat-conductive fluids},
{\it   Proc. Japan Acad. Ser. A Math. Sci.},  {\bf 55} (1979), 337--342.


\bibitem{mic}
 M. Mich\'alek, Stability result for Navier-Stokes equations with entropy transport, {\it J. Math.
Fluid Mech.}, {\bf17} (2015), 279--285.

\bibitem{nash} J. Nash, { Le probl\`{e}me de Cauchy pour les \'{e}quations diff\'{e}rentielles d'un fluide g\'{e}n\'{e}ral},
{\it Bulletin de la Soc. Math. de France},  {\bf 90} (1962), 487--497.





\bibitem{wangchao}
C.~Wang, W. Wang,  Z.~Zhang,
\newblock  {Global well-posedness of compressible Navier-Stokes equations for some classes of
large initial data},
\newblock  {\it  Arch. Ration. Mech. Anal.}, {\bf 213} (2014), 171--214.
\newblock $\,$


\bibitem{zhaixiaopingsiam} X. Zhai, Y. Li, F. Zhou, { Global large solutions  to the three dimensional compressible Navier-Stokes equations}, {\it SIAM J. Math. Anal.}, {\bf 52} (2020), 1806--1843.






\end{thebibliography}
\end{document}